\documentclass[11pt]{amsart}
\usepackage{amsmath,amssymb,amsfonts,url}
\usepackage{amsmath, amsthm, amsfonts, amssymb}
\usepackage{mathrsfs}
\usepackage{bbm}
\usepackage{amssymb}
\usepackage{txfonts}
\numberwithin{equation}{section}

\newtheorem{thm}{Theorem}[section]
\newtheorem{lem}{Lemma}[section]
\newtheorem{rem}{Remark}[section]
\newtheorem{prop}{Proposition}[section]

\newcounter{marnote}

\begin{document}
\title[Monge-Amp\`{e}re equation]{Global solutions and exterior Dirichlet problem for Monge-Amp\`{e}re equation in $\mathbb R^2$} 
\author{Jiguang Bao}
\address{School of Mathematical Sciences\\
Beijing Normal University, Laboratory of Mathematics and Complex Systems\\
Ministry of Education \\
Beijing 100875, China}
\email{jgbao@bnu.edu.cn}

\author{Haigang Li}
\address{School of Mathematical Sciences\\
Beijing Normal University, Laboratory of Mathematics and Complex Systems\\
Ministry of Education \\
Beijing 100875, China}
\email{hgli@bnu.edu.cn}

\author{Lei Zhang}
\address{Department of Mathematics\\
        University of Florida\\
        358 Little Hall P.O.Box 118105\\
        Gainesville FL 32611-8105}
\email{leizhang@ufl.edu}
\thanks{Bao was partially supported by Beijing Municipal Commission of Education for the Supervisor of Excellent Doctoral Dissertation (20131002701), NSFC (11371060) and the Fundamental Research Funds for the Central Universities. Li was partially supported by NSFC (11201029), (11371060) and the Fundamental Research Funds for the Central Universities. }

\date{\today}

\begin{abstract} Monge-Amp\`ere equation $\det(D^2u)=f$ in two dimensional spaces is different in nature from their counterparts in higher dimensional spaces. In this article we employ new ideas to establish two main results for the Monge-Amp\`ere equation defined either globally in $\mathbb R^2$ or outside a convex set. First we prove the existence of a global solution that satisfies a prescribed asymptotic behavior at infinity, if $f$ is asymptotically close to a positive constant. Then we solve the exterior Dirichlet problem if data are given on the boundary of a convex set and at infinity.
\end{abstract}

\maketitle

\section{Introduction}

The aim of this article is to study convex, viscousity solutions of
\begin{equation}\label{main-eq0}
\det(D^2u)=f
\end{equation}
 either globally defined in $\mathbb R^2$ or defined outside a convex set.

The research of global solutions dates back to 1950s. A classical result of J\"orgens (for $n=2$ \cite{jorgens}), Calabi ($n\le 5$ \cite{calabi}), and Pogorelov ($n\ge 2$, \cite{pogorelov1}) states that any classical convex solution of
$$ \det(D^2 u)=1, \quad \mbox{in}~~ \mathbb R^n $$
is a quadratic polynomial. Another proof in the line of affine geometry was given by Cheng-Yau \cite{cheng-yau2}. Caffarelli \cite{caf-liou} gave a proof for viscosity solutions.

If (\ref{main-eq0}) is defined outside a strictly convex, bounded subset in $\mathbb R^n$ and $f\equiv 1$, Caffarelli-Li \cite{caf-li1} proved that the solution $u$ is asymptotically close to a quadratic polynomial at infinity for $n\ge 3$. Similarly for $n=2$ and $f\equiv 1$, using complex analysis Ferrer-Mart\`inez-Mil\'an \cite{FMM1,FMM2} and Delano\"e \cite{delanoe} proved that $u$ is asymptotically close to a quadratic polynomial plus a logarithmic term.

These asymptotics results were extended by the authors in \cite{bao-li-z1} for $f$ being a perturbation of $1$ at infinity. Namely, for $n\ge 3$ and $f$ being an optimal perturbation of $1$, $u$ is asymptotically close to a quadratic polynomial at infinity. For $n=2$ and $f$ being the optimal perturbation of $1$, $u$ is close to a quadratic polynomial plus a logarithmic term at infinity.

Two natural questions are related to the asymptotic behavior of $u$ at infinity. First, given a prescribed asymptotic behavior at infinity, can one find a global solution $u$ that satisfies the asymptotic behavior?  The second question is: Let $D$ be an open, bounded, strictly convex subset of $\mathbb R^n$ with smooth boundary. Given $\phi\in C^{2}(\partial D)$ and a prescribed asymptotic behavior of $u$ at infinity, can one find $u$ of (\ref{main-eq0}) defined in $\mathbb R^n\setminus D$ that satisfies the boundary data at $\partial D$ and infinity?

These questions for $n\ge 3$ are solved in \cite{caf-li1} for $f\equiv 1$ and \cite{bao-li-z1} for $f$ being a perturbation of $1$. However for $n=2$, all the approaches used for higher dimensional cases failed.  The purpose of this article is to employ a new method that solves the existence of global solution for (\ref{main-eq0}) in $\mathbb R^2$ and a corresponding exterior Dirichlet problem.

First we consider convex viscosity solutions of
\begin{equation}\label{main-eq}
\det(D^2u)=f,\quad \mbox{in }~~ \mathbb R^2,
\end{equation}
where we assume $f$ to satisfy
\begin{equation}\label{af}
\begin{cases}
\dfrac 1{c_0}\le f(x)\le c_0,\quad \forall x\in \mathbb R^2,\\
\left|D^j(f(x)-1)\right|\le \dfrac{c_0}{ (1+|x|)^{\beta+j}},\quad j=0,1,..,k, ~~ \forall x\in \mathbb R^2,
\end{cases}
\end{equation}
for some $c_0>0$, $\beta>2$ and $k\ge 3$.
\begin{rem} The assumption $\beta>2$ in (\ref{af}) is sharp, as the readers may see counter examples in the authors' previous work \cite{bao-li-z1}.
\end{rem}

Let $\mathbb M^{2\times 2}$ be the set of the real valued, $2\times 2$ matrices and
$$\mathcal{A}:=\left\{A\in \mathbb M^{2\times 2}:~~ A \mbox{ is symmetric, positive definite and }~ \det(A)=1 \, \right\}. $$
Our first main theorem is on the existence of global solution with prescribed asymptotic behavior at infinity:

\begin{thm}\label{thm1} Suppose (\ref{af}) holds for $f$. Given $A\in \mathcal{A}$, $b\in \mathbb R^2$ and $c\in\mathbb{R}$, there exists $\epsilon_0(A,c_{0})>0$ such that
if
\begin{equation}\label{af-2}
\left|D^m\left(f\Big(\sqrt{A}^{-1}y\Big)-\fint_{\partial B(0,|y|)}f\Big(\sqrt{A}^{-1}x\Big)dS\right)\right|\le \epsilon_0, \quad \forall y\in \mathbb R^2, \quad m=0,1,
\end{equation}
then there exists a unique solution $u$ to
(\ref{main-eq}) satisfying
\begin{equation}\label{asy-1}
\limsup_{|x|\to \infty}|x|^{j+\sigma}\bigg |D^j \left(u(x)-\Big(\frac 12 x'Ax+b\cdot x+d\log \sqrt{x'Ax}+c\Big) \right) \bigg |<\infty
\end{equation}
for $j=0,1,..,k+1$, $\sigma\in (0,\min\{\beta-2,2\})$ and $d=\frac 1{2\pi}\int_{\mathbb R^2}(f-1)$.
\end{thm}

\begin{rem} It is easy to observe that (\ref{af-2}) follows from (\ref{af}) if $|y|$ is large. On the other hand $f_{1}(x):=f\Big(\sqrt{A}^{-1}x\Big)$ could be very different from $1$ when $|x|$ is not large, even though it is very close to a radial function.
\end{rem}

Throughout the article we shall use $B(x_0,r)$ to denote the disk centered at $x_0$ with radius $r$. If $x_0$ is the origin we may use $B_r$.

If the dimension is higher than $2$, the analogue of Theorem \ref{thm1} can be proved using a standard upper-lower solutions method: In order to find a global solution of $\det(D^2u)=f$ for $f$ close to $1$ at infinity, one can solve for
$\det(D^2 u_R)=\bar f$ and $\det(D^2 U_R)=\underline{f}$ in $B_R$, where $\bar f$ and $\underline{f}$ are radial functions greater than $f$ and smaller than $f$ respectively. Both  $\underline{f}$ and $\bar f$ are close to $1$ at infinity and the difference between $u_R$ and $U_R$ is only $O(1)$ if they take the same value on $\partial B_R$. Thus it is easy to obtain a  global solution of $\det(D^2u)=f$ in $\mathbb R^n$ by a sequence of local solutions.
However for $n=2$, such a process is completely destroyed by a logarithmic term. In order for a limiting process to work, it is crucial to obtain a point-wise, uniform estimate for the Hessian matrix of a sequence of approximating solutions. Because of the logarithmic term, the shapes of certain level sets cannot be determined and almost all estimates that work so well for higher dimensional equations fail.

The proof of Theorem \ref{thm1} is as follows. First we look for a radial solution of $\det(D^2u)=\tilde f_1(\,r)$, where $\tilde{f}_1(\,r):=\fint_{\partial{B}_{r}}f_{1}(x)dS$, and take this solution as the first term in our approximation. As we look for more terms down the road we treat the additional terms as solutions to the linearized equation of the Monge-Amp\`ere equation expanded at the radial solution. In order to make all the additional terms proportionally smaller, we need to use the structure of Monge-Amp\`ere equation and a sharp estimate of the Green's function corresponding to the linearized equation. Standard estimates for
Green's functions are not enough for our purpose because the iteration process requires a very sharp form. What makes it worse is the ellipticity of the linearized equation could be very bad near the origin, since $f_1$ could be very different from $1$ near the origin. The proof in Lemma \ref{easy-lem-1}, which relies heavily on results of Kenig-Ni and Cordes-Nirenberg for $n=2$, overcomes this difficulty by estimating the Green's function over ``good regions" first and then use the maximum principle to control the ``bad region".

The second main theorem is on the exterior Dirichlet problem proposed in the previous work of the authors \cite{bao-li-z1}.
We look to solve the following exterior Dirichlet problem: Let $D$ be a bounded, strictly convex set with smooth boundary in $\mathbb R^2$.
Suppose $\varphi\in C^2(\partial D)$ and $u$ is a solution of
\begin{equation}\label{extdir}
\left\{\begin{array}{ll}
\det(D^2u)=f(x),\quad \mbox{in }~~ \mathbb R^2\setminus \overline{D}, \\\\
u\in C^0(\mathbb R^2\setminus D) \mbox{ is a locally convex viscosity solution},\\\\
u=\varphi(x),\qquad \qquad\mbox{on }~~ \partial D.
\end{array}
\right.
\end{equation}

In \cite{bao-li-z1} we conjectured that for any $\varphi\in C^2(\partial D)$, as long as
 $$ d>\frac 1{2\pi}\int_{\mathbb R^2\setminus D}(f-1)-\frac 1{2\pi} \mbox{ area}(D), $$
 there is always a locally convex solution to
 $$
 \begin{cases}
 \det(D^2u)=f(x),\quad \mbox{in }~~ \mathbb R^2\setminus \overline{D}, \\
 u=\varphi(x), \qquad \qquad\mbox{ on }~~ \partial D, \\
 \limsup\limits_{|x|\to \infty}|x|^{j+\sigma}\bigg | D^j\left(u(x)-\Big(\frac 12x'Ax+b\cdot x+d\log \sqrt{x'Ax}+c_d\Big)\right)\bigg |<\infty
\end{cases}
$$
 for $j=0,1,...,k$ ($k\ge 3$), $\sigma\in (0,\min\{\beta-2,2\})$, $c_d\in \mathbb R$ is uniquely determined,
 where $\varphi$ is a given smooth function on $\partial D$, $A\in \mathcal{A}$, $b\in \mathbb R^2$.

Because of the
additional assumption (\ref{af-2}) we are not able to prove this conjecture for arbitrary convex domain $D$. However since we are using a new approach we can weaken the assumption of $\phi$ to being H\"older continuous:

\begin{thm}\label{thm2}
 Let $r_0>0$, $\phi\in C^{\alpha}(\partial B_{r_0})$ for some $\alpha\in (0,1)$ and $f$ satisfy (\ref{af}).
Then for any $d>\frac 1{2\pi}\int_{\mathbb R^2\setminus B_{r_0}}(f-1)-\frac{1}{2}r_0^2$, there exists $\epsilon_0(r_0,d, \alpha)>0$ such that if (\ref{af-2}) holds for $f$ and
$$\sup_{x,y\in \partial B_{r_0}}\frac{|\phi(x)-\phi(y)|}{|x-y|^{\alpha}}\le \epsilon_0, $$
a unique $u$ to (\ref{extdir}) exists ( for $D=B_{r_0}$) and satisfies
\begin{equation}\label{asy-2}
\limsup_{|x|\to \infty}|x|^{j+\sigma}\bigg |D^j\Big(u(x)-(\frac 12 |x|^2+d\log |x|+c_d)\Big) \bigg |<\infty
\end{equation}
for $j=0,..,k+1$ and $\sigma\in (0,\min\{\beta-2,2\})$, $c_d\in \mathbb R$ is uniquely determined by $\phi,d,f$ and $r_0$.
\end{thm}

The organization of this article is as follows. The proof of Theorem \ref{thm1}, which is by an iteration method, is arranged in section two. The proof of Theorem \ref{thm2} in section three is based on a Perron's method. Theorem \ref{thm1} plays an essential role in the proof of Theorem \ref{thm2}. Here we further remark that in order to use Theorem \ref{thm1} in the proof of Theorem \ref{thm2}, it is crucial to assume that $f_1$ is very close to its spherical average rather than $1$.  Finally the proof of Theorem \ref{thm2} also relies on a result (Lemma \ref{uniquenesslem}) of the authors' previous paper \cite{bao-li-z1} to determine the unique constant in the expansion.

\section{Proof of Theorem \ref{thm1}}

Denote
$$f_1(y):=f(\sqrt{A}^{-1}y),\quad\mbox{and}\quad \tilde f_1(y):=\frac{1}{2\pi |y|}\int_{\partial B(0,|y|)}f_1(x)dS.$$
We only need to determine $v(y)$, which satisfies
$$\det(D^2v(y))=f_1(y),\quad y\in \mathbb R^2$$
and
$$\limsup_{|y|\to \infty}|y|^{j+\sigma}\bigg |D^j\Big (v(y)-\frac 12|y|^2-d\log |y|-c\Big )\bigg |=0 $$
for $j=0,...,k+1$ and $\sigma\in (0,\min\{\beta-2,2\})$, where
$$d=\frac 1{2\pi}\int_{\mathbb R^2}(f_{1}-1)dx=\frac 1{2\pi}\int_{\mathbb R^2}(f-1)dx.$$
Once such $v$ is found, we let
$$u(x)=v\left(\sqrt{A}x\right)+b\cdot x. $$
Then we see that (\ref{asy-1}) holds for $u$.

\subsection{Radial solutions and some elementary estimates}

Before we set out to find $v$, we first construct a radial solution of
\begin{equation}\label{rad-u}
\det(D^2 U)=\tilde f_1, \quad \mbox{in}\quad \mathbb R^2.
\end{equation}

Let
$$U(\,r)=\int_0^r\bigg ( \int_0^s 2t\tilde f_1(t)dt   \bigg )^{\frac 12}ds, \quad r=|y|, $$
then one can verify easily that
$$U'(\,r)=\left(\int_0^r 2t\tilde f_1(t)dt\right)^{\frac 12}, \quad U''(\,r)=\frac{r\tilde f_1(\,r)}{\left(\int_0^r 2s\tilde f_1(s)ds\right)^{\frac 12}},$$
and consequently
$$\det(D^2U)=\partial_{11}U\partial_{22}U-\partial_{12}U^2=U''(\,r)\frac{U'(\,r)}r=\tilde f_1(\,r),\quad r>0. $$
Moreover
$$U(\,r)=\frac 12 r^2+d\log r+c_{d}+U(0)+O(r^{-\delta}), \quad \mbox{as }~~ r\to \infty, $$
 where $\delta=\min\{\beta-2,2\}$, using \eqref{af} and the definitions of $\tilde f_1$ and $f_1$,
 $$d=\lim_{r\rightarrow+\infty}\frac{U(\,r)-\frac{r^{2}}{2}}{\log r}=\int_0^{\infty}r\left(\tilde f_1(\,r)-1\right)dr=\frac 1{2\pi}\int_{\mathbb R^2}(f_{1}-1)dx,$$
and
\begin{align*}
c_{d}=&\lim_{r\rightarrow+\infty}\left(U(\,r)-\frac{r^{2}}{2}-d\log\left(r+\sqrt{r^{2}+d}\right)+d\log\frac{r+\sqrt{r^{2}+d}}{r}\right)\\
=&\int_0^{\infty}\bigg ( \Big(\int_0^s 2t f_1(t)dt\Big)^{\frac 12}-s-\frac{d}{\sqrt{s^2+d}}\bigg )ds+d\log 2.
\end{align*}
Note that $f_1$ may not be close to $1$ for $|y|$ not large, but it is close to $\tilde f_1$ when $\epsilon_0$ in (\ref{af-2}) is small.

Next, we will give some estimates for $f_{1}$ and $\tilde{f}_{1}$. We observe that in addition to (\ref{af-2}), $f_1$ also satisfies
\begin{equation}\label{f1-2}
\left\{\begin{array}{ll}
\dfrac{1}{c_0}\le f_1(y)\le c_0,\quad \forall y\in \mathbb R^2,\\
\left|D^j(f_1(y)-1)\right|\le \dfrac{C_0(c_0,A)}{(1+|y|)^{\beta+j}},\quad j=0,1...,k.
\end{array}
\right.
\end{equation}
It is easy to check that in polar coordinates
\begin{equation}\label{af-3}
|\partial_rf_1|+\frac 1r\left|\partial_{\theta}f_1\right|\le \frac{C(c_0,A)}{r^{\beta+1}}, \quad r\ge 1,
\end{equation}
and
\begin{equation}\label{af-4}
|\partial_{rr}f_1|+\frac 1r|\partial_{r\theta}f_1|+\frac 1{r^2}|\partial_{\theta\theta}f_1|\le \frac{C(c_0,A)}{r^{\beta+2}},\quad r\ge 1.
\end{equation}
Now we claim that
\begin{equation}\label{af-5}
|D^j(f_1-\tilde f_1)(y)|\le \frac{C(c_0,A)}{(1+|y|)^{\beta+j}},\quad y\in \mathbb R^2,\quad j=0,1,2.
\end{equation}

Obviously, we just need to verify (\ref{af-5}) for $r=|y|\ge 1$. Indeed, writing $f_1-\tilde f_1$ as
\begin{align}\label{f11}
f_1(y)-\tilde f_1(\,r)
=& f_1(re^{i\psi})-\frac 1{2\pi}\int_0^{2\pi}f_1(re^{i\theta})d\theta \quad \quad (y=re^{i\psi})  \\
=& \frac 1{2\pi}\int_0^{2\pi} \bigg (f_1(re^{i\psi})-f_1(re^{i\theta})\bigg )d\theta. \nonumber
\end{align}
We first use the estimate on $\partial_{\theta}f_1$ in (\ref{af-3}) to obtain
$$|f_1(y)-\tilde f_1(\,r)|\le \frac{C(c_0,A)}{(1+r)^{\beta}}. $$
Then, for $j=1$, we have
$$\left|D(f_1-\tilde f_1)(y)\right|
\le  C\left(|\partial_r f_1|+\frac 1r |\partial_{\theta}f_1|\right)
\le \frac{C(c_0,A)}{(1+r)^{\beta+1}}.
$$
Finally, for $j=2$, it is easy to see from (\ref{f11}) that
$$\left|\partial_{rr}(f_1-\tilde f_1)\right|\le \frac{C(c_0,A)}{(1+r)^{\beta+2}}. $$
Since $\tilde f_1$ is radial,
$$\partial_{r\theta}(f_1-\tilde f_1)=\partial_{r\theta}f_1, \quad \partial_{\theta\theta}(f_1-\tilde f_1)=\partial_{\theta\theta}f_1. $$
Therefore, by (\ref{af-4}),
\begin{align*}
&\left|D^2(f_1-\tilde f_1)(x)\right|
\le
C\left(\left|\partial_{rr}(f_1-\tilde f_1)\right|+\frac{|\partial_{r\theta}f_1|}r+\frac{|\partial_{\theta\theta}f_1|}{r^2}\right)
\le  \frac{C}{(1+r)^{\beta+2}}.
\end{align*}
Thus, (\ref{af-5}) is established. Combining (\ref{af-2}) and (\ref{af-5}), we obtain
\begin{equation}\label{f-af}
\left|D^m (f_1-\tilde f_1)(y)\right|\le \frac{\epsilon_1(\epsilon_0,A,c_0,\beta)}{ (1+|y|)^{\beta_1}},\quad y\in \mathbb R^2, ~~m=0,1,
\end{equation}
where $\beta_1=\frac{\beta}2+1\in (2,\beta)$ and $\epsilon_1\to 0$ as $\epsilon_0\to 0$.

We further obtain, by simple computations, that
\begin{equation}\label{Uij}
\partial_{11}U=F_1+F_2\cos (2\theta),\quad \partial_{22}U=F_1-F_2\cos(2\theta),\quad \partial_{12}U=F_2\sin(2\theta)
\end{equation}
where
$$F_1:=\frac 12(U''(\,r)+U'(\,r)/r),\quad F_2:=\frac 12(U''(\,r)-U'(\,r)/r). $$
It follows from (\ref{f1-2}), (\ref{af-2}) and (\ref{af-5}) that there exists $c_1(c_0,A)>0$ such that
\begin{equation}\label{close-delta}
\begin{cases}
\left|D^j\left(\partial_{22}U-1\right)(y)\right|\le \dfrac{c_1}{ (1+|y|)^{2+j}}, \quad y\in \mathbb R^2, \\\\
\left|D^j(\partial_{11}U-1)(y)\right|\le \dfrac{c_1}{ (1+|y|)^{2+j}},\quad y\in \mathbb R^2, \\\\
\left|D^j(\partial_{12}U)(y)\right|\le \dfrac{c_1}{ (1+|y|)^{2+j}},\quad y\in \mathbb R^2,
\end{cases}
\end{equation}
for $j=0,1,2$. It is easy to verify (\ref{close-delta}) for $y$ large since $\tilde f_1$ is close to $f_1$ and $f_1$ is close to $1$ when $|y|$ is large. For $|y|$ not large (\ref{close-delta}) certainly holds.

\subsection{The first step of iteration}

Suppose that the solution $u$ of (\ref{main-eq}) is of the form
$$u=U+\phi.$$ Clearly $\phi$ satisfies
\begin{equation}\label{eq-phi}
\partial_{11}\phi\partial_{22}U+\partial_{22}\phi\partial_{11}U-2\partial_{12}\phi\partial_{12}U+\det(D^2\phi)=f_1-\tilde f_1,\quad \mbox{in}~~ \mathbb R^2.
\end{equation}

Let
$$a^*_{11}:=\partial_{22}U,\quad a^*_{22}:=\partial_{11}U,\quad a^*_{12}:=-\partial_{12}U,$$
then by \eqref{close-delta},
$$c_1^{-1} I\le (a^*_{ij})_{2\times 2}\le c_1 I. $$
It is well known that the first part of (\ref{eq-phi}) can be written as a divergence form.
$$L\phi:=\partial_i(a^*_{ij}\partial_j \phi)=\partial_{22}U\partial_{11}\phi+\partial_{11}U\partial_{22}\phi-2\partial_{12}U\partial_{12}\phi,\quad \forall \phi\in C^2(\mathbb R^2), $$
because $\partial_i a^*_{ij}=0$ for $j=1,2$.  Then \eqref{eq-phi} can be written as
\begin{equation}\label{eq-phi-1}
\partial_i(a^{*}_{ij}\partial_j \phi)+\det(D^2\phi)=f_1-\tilde f_1,\quad \mbox{in}~~ \mathbb R^2.
\end{equation}
Let $G$ be the fundamental solution of $-L$ on $\mathbb R^2$
$$-\partial_{y_i}(a_{ij}^*(y)\partial_{y_j}G(x,y))=\delta_x, \quad \mbox{ in }~~ \mathbb R^2, $$
where $\delta_x$ is the Dirac mass at $x$. According to the theory of Kenig-Ni \cite{kenig} there exists $c_2(c_0,A)$ such that
\begin{equation}\label{green}
|G(x,y)|\le \left\{\begin{array}{ll}
c_2\big|\log |x-y| \big|,\quad y\in B(x,\frac 12),\\\\
c_2\Big(\big|\log |x-y| \big|+1\Big)\quad y\in \mathbb R^2\setminus B(x,\frac 12).
\end{array}
\right.
\end{equation}

In the following, we will start our iteration process.  We first solve
\begin{equation}\label{phi0-iteration}
L\phi^0=f_1-\tilde f_1,\quad \mbox{in}~~  \mathbb R^2
\end{equation}
by letting
\begin{equation}\label{phi0-est}
\phi^0(x)=\int_{\mathbb R^2}G(x,y)(\tilde f_1(y)-f_1(y))dy.
\end{equation}

The estimates of $\phi^0$ are stated in the following. The proof will be given in subsection \ref{proofofproposition}.
\begin{prop}\label{phi0-prop}
There exists $c_3>0$ only depending on $c_0,A$ and $\beta$ such that $\phi^0$ satisfies
\begin{equation}\label{phi0-2}
\begin{cases}
\left|D^j\phi^0(x)\right|\le \dfrac{c_3\epsilon_1}{(1+|x|)^{j+\tau}},\quad \forall x\in \mathbb R^2,\, j=0,1,2\\\\
\left|D^2\phi^0(y)-D^2\phi^0(z)\right|\le c_3\epsilon_1|y-z|^{\alpha},\quad \forall~y,z\in{B}_{1}, \\\\
\left|D^2\phi^0(y)-D^2\phi^0(z)\right|\le \dfrac{c_3\epsilon_1}{ |x|^{2+\tau+\alpha}}|y-z|^{\alpha}, ~~\forall~ y,z\in B_{\frac{3|x|}{2}}\setminus B_{\frac{|x|}{2}},~~ |x|>1,
\end{cases}
\end{equation}
where $\tau\in\left(0,\frac{\beta}2-1\right)$,$\alpha\in (0,1)$ depends on $c_0,A,\beta$.
\end{prop}

Once we have the estimate for $\phi_{0}$ from Proposition \ref{phi0-prop}, we let $$\psi^1(x)=\int_{\mathbb R^2}G(x,y)\det(D^2\phi^0(y))dy, $$
then $\psi^1$ solves
\begin{equation}\label{psi1-iteration}
L\psi^1=-\det(D^2\phi^0),\quad \mbox{in}~~ \mathbb R^2.
\end{equation}
Since
$$\det(D^2\phi^0)=\partial_1\left(\partial_1 \phi^0\partial_{22}\phi^0\right)-\partial_2\left(\partial_{12}\phi^0\partial_1\phi^0\right),$$
we write $\psi^1$ as
$$\psi^1(x)=\int_{\mathbb R^2}\left(-\partial_{y_1}G(x,y)\partial_{1}\phi^0(y)\partial_{22}\phi^{0}(y)+\partial_{y_2}G(x,y)\partial_{1}\phi^0(y)\partial_{12}\phi^0(y)\right)dy. $$

It is easy to use the decay rate of $D^2\phi^0$ in (\ref{phi0-2}) to obtain
\begin{equation}\label{psi-11}
\left|\psi^1(x)\right|\le C(c_0,A,\beta)(c_3\epsilon_1)^2,\quad x\in B_{2R_0}.
\end{equation}
Then from (\ref{psi-11}) and elliptic estimate we have
\begin{equation}\label{psi-12}
\left\|\psi^1(x)\right\|_{C^{2,\alpha}(B_{R_0})}\le  C(c_0,A,\beta) c_3^2\epsilon_1^2.
\end{equation}
For $|x|>R_0$, we decompose $\mathbb R^2$ into $E_1\cup E_2$. For the integral on $E_1=B(0,\frac{|x|}{2})$, we use Proposition \ref{phi0-prop} to get
\begin{align*}
&\left|\int_{E_1}\left(\partial_{y_1}G(x,y)\partial_{1}\phi^{0}(y)\partial_{22}\phi^{0}(y)-\partial_{y_2}G(x,y)\partial_{1}\phi^{0}(y)\partial_{12}\phi^{0}(y)\right)dy\right|  \\
\le & C(c_0,\beta,A)(c_3\epsilon_1)^2\frac{\log |x|}{|x|^{2+2\tau}}\le \frac{C(c_0,A,\beta)(c_3\epsilon_1)^2}{(1+|x|)^{\tau}}.
\end{align*}

\begin{rem} Writing $\det(D^2\phi^0)$ in the divergence form leads to differentiation on $G$ and thus we avoid a logarithmic term from the integration over $E_1$.
This is exactly like the corresponding estimate for $\phi^0$. Here we further remark that the estimate for $\psi^1$ is exactly like that for $\phi^0$, as the estimate of $G$ is the same, the H\"older norm of the elliptic operator in the scaling part still has the same bound.
\end{rem}
Using the rough estimate of $G$, \eqref{green}, and estimates of $\phi^{0}$, we obtain easily
\begin{align*}
&\left|\int_{E_2}\left(\partial_{y_1}G(x,y)\partial_{1}\phi^{0}(y)\partial_{22}\phi^{0}(y)-\partial_{y_2}G(x,y)\partial_{1}\phi^{0}(y)\partial_{12}\phi^{0}(y)\right) dy \right|\\
\le & \frac{C(c_0,A,\beta)(c_3\epsilon_1)^2}{|x|^{4+2\tau}}
\le \frac{C(c_0,\beta,A)(c_3\epsilon_1)^2}{(1+|x|)^{\tau}}.
\end{align*}
Correspondingly elliptic estimates lead to estimates on higher derivatives. Therefore the following estimates have been obtained for $\psi^1$: for $x\in \mathbb R^2$,
there exists $c_4(c_0,\beta,A)>0$ such that
\begin{equation}\label{psi-1e}
\begin{cases}
\left|D^j\psi^1(x)\right|\le \dfrac{c_4 c_3^2 \epsilon_1^2}{(1+|x|)^{j+\tau}},\quad \forall x\in \mathbb R^2,\, j=0,1,2\\\\
\left|D^2\psi^1(y)-D^2\psi^1(z)\right|\le c_4 c_3^2 \epsilon_1^2|y-z|^{\alpha},\quad \forall~y,z\in{B}_{1}, \\\\
\left|D^2\psi^1(y)-D^2\psi^1(z)\right|\le \dfrac{c_4 c_3^2 \epsilon_1^2}{ |x|^{2+\tau+\alpha}}|y-z|^{\alpha}, ~~\forall~ y,z\in B_{\frac{3|x|}{2}}\setminus B_{\frac{|x|}{2}},~~ |x|>1,
\end{cases}
\end{equation}
where $\alpha\in (0,1)$ is defined as in (\ref{phi0-2}).

\begin{rem} The constant $c_4$ in (\ref{psi-1e}) only depends on $c_0,\beta,A$ and is obtained from evaluating the Green's representation formula and standard elliptic estimates. If the $\det(D^2\phi^0)$ is replaced by another function with fast decay at infinity, the constant $c_4$ does not change.
\end{rem}

\subsection{Completion of the proof of Theorem \ref{thm1} by iteration}

\begin{proof}[Proof of Theorem \ref{thm1}] We will prove it by iteration. Let
$$\phi^1:=\phi^0+\psi^1,$$
then, it is clear from \eqref{phi0-iteration} and \eqref{psi1-iteration} that
\begin{equation}\label{phi1-iteration}
L\phi^1=L\phi^0+L\psi^1=f_1-\tilde f_1-\det(D^2\phi^0).
\end{equation}
Rewrite it as
$$L\phi^1+\det(D^2\phi^1)=f_1-\tilde f_1+\det(D^2\phi^1)-\det(D^2\phi^0). $$
Let $\psi^2$ solve
$$L\psi^2:=\det(D^2\phi^0)-\det(D^2\phi^1). $$

In general, for $l\geq2$, we define
$$\phi^{l}:=\phi^{l-1}+\psi^{l}, $$
and
$$L\psi^{l}:=\det(D^2 \phi^{l-2})-\det(D^2\phi^{l-1}).$$
We will prove  the following estimates for $\phi^{l}$, $l\geq0$:
\begin{equation}\label{phi-e}
\left\{\begin{array}{ll}
\left|D^j\phi^l(y)\right|\le \dfrac{2c_3 \epsilon_1}{ (1+|y|)^{\tau+j}},\quad y\in \mathbb R^2, \,\, j=0,1,2\\
\\
\left\|\phi^l\right\|_{C^{2,\alpha}(B_1)}\le 2c_3\epsilon_1, \\
\\
|D^2\phi^l(y)-D^2\phi^l(z)|\le \frac{2c_3\epsilon_1}{|x|^{\tau+2+\alpha}}|y-z|^{\alpha}, \,\, y, z\in B(x,\frac{|x|}2), |x|>1.
\end{array}
\right.
\end{equation}
by using the following estimates for $\psi^{l}$, $l\geq0$,
\begin{equation}\label{psi-e}
\begin{cases}
\left|D^j\psi^{l+1}(x)\right|\le \dfrac{2c_4 (c_3\epsilon_1)^{l+2}}{(1+|x|)^{j+\tau}},\quad \forall x\in \mathbb R^2,\, j=0,1,2\\\\
\left|D^2\psi^{l+1}(y)-D^2\psi^l(z)\right|\le 2c_4 (c_3\epsilon_1)^{l+2}|y-z|^{\alpha},\quad \forall~y,z\in{B}_{1}, \\\\
\left|D^2\psi^{l+1}(y)-D^2\psi^l(z)\right|\le \dfrac{2c_4 (c_3\epsilon_1)^{l+2}}{ |x|^{2+\tau+\alpha}}|y-z|^{\alpha}, ~~\forall~ y,z\in B_{\frac{3|x|}{2}}\setminus B_{\frac{|x|}{2}},~~ |x|>1,
\end{cases}
\end{equation}
which can be proved by induction.

First, for $l=0$, we have from \eqref{phi0-2} and \eqref{psi-1e} that \eqref{phi-e} and \eqref{psi-e} holds, respectively.
Then, by the definition of $\phi^{1}$, $\phi^1=\phi^0+\psi^1$, using the estimate of $\phi^0$ and $\psi^1$, we immediately have
$$|D^j\phi^1(y)|\le |D^j\phi^0(y)|+|D^j\psi^1(y)|\leq \frac{(c_3\epsilon_1+c_4c_3^2\epsilon_1^2)}{ (1+|y|)^{\tau+j}}, $$
for $y\in \mathbb R^2$ and $j=0,1,2$. The $C^{\alpha}$ estimate for the second derivatives are similar. If we choose $\epsilon_1$ to satisfy $c_4c_3\epsilon_1<\frac 12$ and $c_3\epsilon_1<\frac 12$,  then we obtain the estimate \eqref{phi-e} holds for $\phi^1$.

Since $\psi^2$ solve the linear equation, it has the expression
\begin{align*}
\psi^2(y):&=\int_{\mathbb R^2}G(y,\eta)(\det(D^2\phi^1)-\det(D^2\phi^0))d\eta\\
&=\int_{\mathbb R^2}\partial_{\eta_1}G(y,\eta)\left(-\partial_{1}\phi^{1}\partial_{22}\phi^{1}+\partial_{1}\phi^{0}\partial_{22}\phi^{0}\right)\\
&\qquad\qquad+\partial_{\eta_2}G(y,\eta)
\left(-\partial_{1}\phi^{0}\partial_{12}\phi^{0}+\partial_{1}\phi^{1}\partial_{12}\phi^{1}\right)d\eta.
\end{align*}
It is easy to see
\begin{align*}
\partial_{1}\phi^{1}\partial_{22}\phi^{1}-\partial_{1}\phi^{0}\partial_{22}\phi^{0}=\partial_{1}\phi^{0}\partial_{22}\psi^{1}+\partial_{1}\psi^{1}\partial_{22}\phi^{0}+\partial_{1}\psi^{1}\partial_{22}\psi^{1},\\
\partial_{1}\phi^{1}\partial_{12}\phi^{1}-\partial_{1}\phi^{0}\partial_{12}\phi^{0}=\partial_{1}\phi^{0}\partial_{12}\psi^{1}+\partial_{1}\psi^{1}\partial_{12}\phi^{0}+\partial_{1}\psi^{1}\partial_{12}\psi^{1}.
\end{align*}
Thus $\psi^2$ can be evaluated as
\begin{align*}
\psi^2(y)=\int_{\mathbb R^2}\Big(&-\partial_{\eta_1}G(y,\eta)\left(\partial_{1}\phi^{0}\partial_{22}\psi^{1}+\partial_{1}\psi^{1}\partial_{22}\phi^{0}+\partial_{1}\psi^{1}\partial_{22}\psi^{1}\right)\\
&+\partial_{\eta_2}G(y,\eta)\left(\partial_{1}\phi^{0}\partial_{12}\psi^{1}+\partial_{1}\psi^{1}\partial_{12}\phi^{0}+\partial_{1}\psi^{1}\partial_{12}\psi^{1}\right)\Big)d\eta.
\end{align*}
Using (\ref{phi0-2}) and (\ref{psi-1e}) we obtain \eqref{psi-e} holds for $\psi^{2}$. That is, \eqref{psi-e} holds for $l=1$.

Suppose that \eqref{phi-e} and \eqref{psi-e} holds for $l=k$, then by
$$\phi^{k+1}:=\phi^{k}+\psi^{k+1}=\phi^{0}+\sum_{l=1}^{k}\psi^{l},$$
we have
\begin{align*}
\left|D^j\phi^{k+1}(y)\right|
&\leq \left|D^j\phi^{0}(y)\right|+\sum_{l=1}^{m}\left|D^{j}\psi^{l}\right|\\
&\le \frac{c_3\epsilon_1+ c_4(c_3\epsilon_1)^2+ 2c_4(c_3\epsilon_1)^3+\cdots+2c_4(c_3\epsilon_1)^{l+1}}{(1+|y|)^{j+\tau}}\\
&\le \frac{c_3\epsilon_1\left(1+ c_4(c_3\epsilon_1)+ 2c_4(c_3\epsilon_1)^2+\cdots+2c_4(c_3\epsilon_1)^{l}\right)}{(1+|y|)^{j+\tau}}\\
&\le 2c_{3}\epsilon_{1}(1+|y|)^{-j-\tau},\quad j=0,1,2,\quad \mbox{in }~~ \mathbb R^2.
\end{align*}
Similarly, we have \eqref{phi-e} holds for $\phi^{k+1}$. Continue this process, we can obtain \eqref{phi-e} and \eqref{psi-e} holds for any $l\geq0$.

Notice that for all $l$, the estimates of $\phi^l$ satisfy the same bound as in (\ref{phi-e}), because the estimates for $\psi^l$ use the same estimate for $G$ and $D G$. The only difference is the right hand side: $\det(D^2\phi^l)-\det(D^2\phi^{l+1})$. Thus, for $\epsilon_1$ small the process converges and $\phi^{l}$ converges to a solution of
$$\det(D^2 v)=f. $$
The estimates on the asymptotic behavior of $u$ at infinity as well as their derivatives can be determined by the main theorem in \cite{bao-li-z1}.
Theorem \ref{thm1} is established.
\end{proof}

\subsection{Proof of Proposition \ref{phi0-prop}}\label{proofofproposition}

From (\ref{close-delta}) we see that
\begin{equation}\label{close-1}
 |D^j (a^*_{ij}-\delta_{ij})(y)|\le c_2 (1+|y|)^{-2-j},\quad j=0,1,2, \quad \forall y\in \mathbb R^2.
 \end{equation}
So $a_{ij}^*$ is very close to $\delta_{ij}$ when $|y|$ is large.

Before we present the proof of Proposition \ref{phi0-prop} we list two tools needed for this proof: Cordes-Nirenberg estimate and an estimate of the Green's function of $L$. The Cordes-Nirenberg estimate is stated in the following lemma (see e.g. \cite{cabre}):

\begin{lem}(Cordes-Nirenberg) \label{cordes}
For any $h$ satisfying
$$ a_{ij}\partial_{ij}h=0,\quad \mbox{in }~~ B_1\subset \mathbb R^n,\quad n\ge 2, $$
there exists an $\delta_0>0$ depending only on $n$ such that if $| a_{ij}-\delta_{ij}|\le \delta_0$ for all $i,j=1,...,n$
the following estimate holds:
$$\|D h\|_{C^{1/2}(B_{1/2})}\le C(n)\|h\|_{L^{\infty}(B_1)}. $$
\end{lem}

The second tool is a gradient estimate of $G(x,y)$ for $|x|>2R_0$ and $|y|\le |x|/2$. Here $R_0(c_0,\beta)$ is a large number that satisfies the following requirement: For any $R>R_0$, let
$$a_{ij}^R(y):=a_{ij}^*(Ry),\quad \frac 12<|y|<2, \quad i,j=1,2 $$
there holds
\begin{equation}\label{est-aij}
|a_{ij}^R(y)-\delta_{ij}|\le \delta_0 \quad \mbox{ and } \quad \|a_{ij}^R(\cdot)\|_{C^{\alpha}(B_2\setminus B_{1/2})}\le 4.
\end{equation}
where $\delta_0$ is the absolute constant required in the Cordes-Nirenberg estimate. It is easy to see that (\ref{est-aij}) holds from (\ref{close-1}) for $R_0$ large that only depends on $c_0$, $\beta$ and $A$.

\begin{lem}\label{easy-lem-1}
For $|x|>2R_0$, there exists $C(\beta,c_0,A)>0$ such that
$$|D_{y}G(x,y)|\le C(\beta,c_0,A)\frac{\log |x|}{|x|},\quad \forall y\in B(0,\frac{|x|}2). $$
Here $D_{y}$ means the differentiation with respect to the component $y$.
\end{lem}

\begin{proof}Let $g(y):=G(x,y)$ for $|y|<\frac 9{10}|x|$ and we write the equation for $g$ in $B(0,\frac 9{10}|x|)$ as
\begin{equation}\label{eq-F}
a_{ij}^*\partial_{ij}g=0,\quad \mbox{in }~~ B(0,\frac 9{10}|x|).
\end{equation}
we first estimate $|Dg|$ over $B(0,\frac 34|x|)\setminus B(0,\frac 12|x|)$.  For any fixed $y$ in this region, let
$R=\frac 1{10}|x|$ and
$$\bar{a}_{ij}^R(z):=a_{ij}^*(y+Rz),\quad g_R(z):=g(y+Rz),\quad |z|\le 1. $$
Clearly $|g_R(z)|\le C\log |x|$ by the estimate of Kenig-Ni and
 $$\bar{a}_{ij}^R(z)\partial_{z_iz_j}g_R(z)=0,\quad \mbox{in }~~ B_1. $$
By the definition of $R_0$, we have
$|\bar{a}_{ij}^R-\delta_{ij}|\le \delta_0$ where $\delta_0$ is small enough for Lemma \ref{cordes} to be applied. Using $|g_R(z)|\le C\log |x|$ and Lemma \ref{cordes} we have
$$ |D g_R(z)|\le C\log |x|, \quad z\in B_{1/2}, $$
which gives
\begin{equation}\label{out-G}
|Dg|\le \frac{C\log |x|}{|x|},\qquad \frac{9}{20}|x|\le |y|\le \frac 45|x|.
\end{equation}

Now let
$$H(y):=\partial_{1}g(y_{1},y_{2}) , \quad y=(y_{1},y_{2})\in B(0,\frac{|x|}2). $$
Differentiating (\ref{eq-F}) with respect to $y_1$:
\begin{equation}\label{for-G}
a_{ij}^*\partial_{ij}H+\partial_1a_{11}^*\partial_1 H+2\partial_1a_{12}^*\partial_2 H+\partial_1 a_{22}^*\partial_{22}F=0, \quad \mbox{in }~~ B(0,\frac 12|x|).
\end{equation}
Using (\ref{eq-F}) again for the last term of (\ref{for-G}), we have
\begin{equation}\label{for-G-2}
\partial_{22}g=-\frac{a_{11}^*\partial_{11}g+2a_{12}^*\partial_{12}g}{a_{22}^*}.
\end{equation}
Combining (\ref{for-G}) and (\ref{for-G-2}) we have
$$
a_{ij}^*\partial_{ij}H+\left(\partial_{1}a_{11}^*-\frac{\partial_{1}a_{22}^*}{a_{22}^*}a_{11}^*\right)\partial_1H+\left(2\partial_1a_{12}^*-\frac{2a_{12}^*}{a_{22}^*}\partial_1a_{22}^*\right)\partial_2H=0 $$
in $B(0,\frac 12|x|)$. Clearly maximum principle holds for $H$ and it gives the desired bound for $H$. The estimate of $\partial_{2}g(y)$ for $y\in B(0, |x|/2)$ is similar.
Lemma \ref{easy-lem-1} is established.
\end{proof}

\begin{proof}[Proof of Proposition \ref{phi0-prop}]

The estimate of $\phi^0$ consists of two cases: $x\in B_{R_0}$ and $x\in \mathbb R^2\setminus B_{R_0}$.

First for $x\in B_{R_0}$, it is easy to use
(\ref{green}) and (\ref{f-af}) in (\ref{phi0-est}) to obtain
$$|\phi^0(x)|\le \epsilon_1 C(c_0,\beta,A),\quad \mbox{ for }~~ |x|<R_0. $$
The estimates for higher derivatives of $\phi^0$ in $B_{R_0}$ follow by standard elliptic estimate. Thus (\ref{phi0-2}) is verified in $B_{R_0}$.

For the second case: $x\in \mathbb R^2\setminus B_{R_0}$, we integrate over $E_1=B(0, |x|/2)$ and $E_2=\mathbb R^2\setminus E_1$, respectively. The integration over $E_1$ can be written as
\begin{align*}
&\left|\int_{E_1}(G(x,y)-G(x,0))(\tilde f_1- f_1)dy\right|
\le  \int_{E_1} \left|D_2 G(x,\xi)\right| \cdot |y| \cdot\left|f_1(y)-\tilde f_1(|y|)\right|dy,
\end{align*}
where $\xi$ is on the segment $oy$, because the integration of $f_1-\tilde f_1$ over $E_1$ is zero. By Lemma \ref{easy-lem-1} the integration over $E_1$ is bounded by $C(c_0,\beta,A)\epsilon_1|x|^{2-\beta_1}\log |x|$.
The integration over $E_2$ can be estimated by the rough bound of $G(x,\eta)$ and $f_1-\tilde f_1$. Then one sees easily that the bound for this part is
$C(\beta,c_0,A)\epsilon_1 |x|^{2-\beta_1}\log |x|$. Consequently for all $x\in \mathbb R^2$, we have
\begin{equation}\label{phi0-3}
|\phi^0(x)|\le C(c_0,A,\beta)\epsilon_1|x|^{2-\beta_1}\log |x|\le \frac{C(c_0,\beta,A)\epsilon_1}{ (1+|x|)^{\tau}},
\end{equation}
for $\tau\in(0,\frac{\beta}{2}-1)$. (\ref{phi0-2}) is established for $j=0$.

To prove (\ref{phi0-2}) for $j\ge 1$ and $|x|>R_0$,
we apply the following re-scaling argument:  consider
$$ \phi^0_{R}(y):=\phi^0(Ry),\quad \frac 14\le |y|\le 2,\quad R=|x|>R_0. $$
Then direct computation gives
$$\partial_i\left(a^*_{ij}(Ry)\partial_j\phi^0_{R}(y)\right)=R^2\left(f_1(Ry)-\tilde f_1(Ry)\right),\quad \mbox{in}~~B_2\setminus B_{1/4}. $$
The $C^{1}$ norm of the right hand side is $O(R^{2-\beta})$ and the coefficients $a^*_{ij}(Ry)$ is only $O(R^{-2})$ different from $\delta_{ij}$ in $C^1$ norm as well. Moreover,
 by (\ref{phi0-3}), $\left|\phi_R^0\right|\le C\epsilon_1 R^{-\tau}$ in $B_2\setminus B_{1/4}$. Thus
 standard elliptic estimate gives
\begin{align*}
 \left\|\phi^0_R\right\|_{C^{2,\alpha}(B_{3/2}\setminus B_{1/2})}
 \le &C(c_0,A,\beta) \bigg (\sup_{B_{2}\setminus B_{1/4}}\left|\phi_R^0\right|
 +\left\|R^2(f_1-\tilde f_1)(R\cdot)\right\|_{C^{\alpha}(B_{3/2}\setminus
B_{1/2})} \bigg ) \\
\le &\frac{C(c_0,A,\beta)\epsilon_1}{R^{\tau}}.
\end{align*}
Proposition \ref{phi0-prop} follows from the estimate above.
\end{proof}

\begin{rem} The use of $\tilde f_1$ is quite essential in the estimate over $E_1$. Otherwise a logarithmic term will occur.
\end{rem}

\section{Proof of Theorem \ref{thm2}}

Recall that the assumption on $d$ is
 $$d>\frac 1{2\pi}\int_{\mathbb R^2\setminus B_{r_0}}(f-1)-\frac 12 r_0^2.$$
 By choosing $\epsilon_0$ sufficiently small, depending on $r_0$ and $d$, we can extend $f$ to the whole $\mathbb R^2$ such that $f$ satisfies (\ref{af}), (\ref{af-2}) and
 $$d=\frac 1{2\pi}\int_{\mathbb R^2}(f-1).$$
By Theorem \ref{thm1} we can find $U$ to satisfy
$$
\begin{cases}
\det(D^2U)=f, \quad \mbox{ in }\quad \mathbb R^2, \\
U(x)=\frac 12|x|^2+d\log |x|+C+O(|x|^{-\sigma}),\quad |x|>1 \\
U \mbox{ is  close to a radial function }.
\end{cases}
$$
 By adding a constant to $U$ if necessary we can make
\begin{equation}\label{bry-s}
\|\phi-U\|_{C^{\alpha}(\partial B_{r_0})}\le \epsilon_1(\epsilon_0).
\end{equation}
where $\epsilon_1>0$ depends on $\epsilon_0$ and tends to $0$ as $\epsilon_0\to 0$.

Now we look for a function $u=U+h$ to satisfy
$$
\begin{cases}
\det(D^2u)=f,\quad \mbox{in}\quad \mathbb R^2\setminus B_{r_0},\\
u=\varphi,\quad \mbox{ on }\quad \partial B_{r_0}\\
u=\frac 12|x|^2+d\log |x|+O(1),\quad |x|>1.
\end{cases}
$$
Using the information of $U$ we need to find $h$ to satisfy
\begin{equation}\label{h-eq}
\begin{cases}
\partial_i(a_{ij}\partial_jh)+\det(D^2h)=0, \quad \mbox{in }\quad \mathbb R^2\setminus B_{r_0}, \\
h=\varphi-U, \quad \mbox{ on }\quad \partial B_{r_0},\\
h=O(1), \quad \mbox{in }\quad |x|>r_0.
\end{cases}
\end{equation}
where $a_{11}=U_{22}$, $a_{22}=U_{11}$, $a_{12}=-U_{12}$. Just like in the proof of Theorem \ref{thm1} we have
$$|D^m(a_{ij}(x)-\delta_{ij})|\le C|x|^{-2-m},\quad m=0,1,2. $$

For the remaining part of the proof we shall use
$$L=\partial_i(a_{ij}\partial_j)=a_{ij}(x)\partial_{x_ix_j}.  $$
  We first look for $\psi_0$ that satisfies
$$
\begin{cases}
L\psi_0=0, \quad \mbox{in }\quad \mathbb R^2\setminus B_{r_0},\\
\psi_0=\varphi-U,\quad \mbox{on }\quad \partial B_{r_0}, \\
|\psi_0|\le \epsilon_1,\quad |D^j\psi_0|\le C\epsilon_1 |x|^{-2-j},\quad j=1,2,3.
\end{cases}
$$
The function $\psi_0$ can be determined as follows: Let $y=x/|x|^2$ for $|x|>r_0$ and $|y|<r_0$. Let $\tilde \psi_0(y)=\psi_0(y/|y|^2)$. Direct computation yields
$$b_{kl}(y)\partial_{y_ky_l}\tilde \psi_0+b_k(y)\partial_{y_k}\tilde \psi_0=0, \quad \mbox{ in } \quad B_{1/r_0} $$
where
$$b_{kl}=\frac{1}{|y|^4}\frac{\partial y_k}{\partial x_i}a_{ij}(\frac{y}{|y|^2})\frac{\partial y_l}{\partial x_j}
=(\delta_{ki}-2\frac{y_ky_i}{|y|^2})a_{ij}(\frac{y}{|y|^2})(\delta_{lj}-2\frac{y_ly_j}{|y|^2}), $$
and
$$b_k(y)=a_{ij}(\frac{y}{|y|^2})\frac{2\delta_{ki}y_l-2\delta_{kl}y_i-2y_k\delta_{il}}{|y|^2}(\delta_{lj}-\frac{2y_ly_j}{|y|^2}). $$
Because of the closeness between $a_{ij}$ and $\delta_{ij}$ one verifies easily that $b_{kl}$ is uniformly elliptic in $B_{1/r_0}$ and the $C^{\alpha}$ norm of
both $b_{kl}$ and $b_k$ in $B_{1/r_0}$ is finite.

By Schauder's estimate
$$\|\tilde \psi_0\|_{C^{2,\alpha}(B_{1/r_0})}\le c_1(c_0,d,r_0)\epsilon_1. $$
Thus by the definition of $\tilde \psi_0$ and standard elliptic estimate
$$|D^m\psi_0(x)|\le C\epsilon_1 |x|^{-2-m}\quad m=0,1,2,3 \quad |x|>r_0. $$
Next we solve
$$
\begin{cases}
L\psi_1=-\det(D^2\psi_0),\quad \mbox{in }\quad |x|>r_0 \\
\psi_1=0,\quad \mbox{ on }\quad \partial B_{r_0},\quad \psi_1=O(1) \mbox{ at } \quad \infty.
\end{cases}
 $$
by the reflection method.  Using the smallness of $\psi_0$
we have
$$|D^m\psi_1(x)|\le c_1(c_1\epsilon_1)^2|x|^{-2-m}=c_1^3\epsilon_1^2|x|^{-2-m},\quad m=0,1,2,3,\quad |x|> r_0. $$
Let $h_0=\psi_0$ and $h_1=\psi_1+\psi_0$. Then it is easy to see that $h_1$ satisfies
$$L h_1+ det(D^2 h_0)=0,\quad |x|>r_0. $$

Then we move on to define
$$
\begin{cases}
L\psi_2=\det(D^2h_0)-\det(D^2h_1), \quad |x|>r_0, \\
\psi_2=0,\quad \mbox{ on }\quad \partial B_{r_0},\quad \psi_2=O(1) \mbox{ at infinity}.
\end{cases}
$$
Based on the estimates on $h_0$ and $h_1$ we have
$$|D^m\psi_2(x)|\le c_1^5\epsilon_1^3|x|^{-2-m},\quad m=0,1,2,3,\quad |x|>r_0. $$
Let $h_2=h_1+\psi_2$. Then it is easy to verify that
$$Lh_2+\det(D^2h_1)=0, \quad |x|>r_0. $$
In general we determine $\psi_k$ to satisfy
$$\left\{\begin{array}{ll}
L \psi_k=\det(D^2 h_{k-2})-\det(D^2h_{k-1}), \quad |x|>r_0, \\
\psi_k=0,\quad \mbox{ on }\quad \partial B_{r_0}, \quad \psi_k=O(1) \,\, \mbox{ at } \infty.
\end{array}
\right.
$$
For $\psi_k$ we have
$$|D^m\psi_k(x)|\le c_1^{2k+1}\epsilon_1^k|x|^{-2-m},\quad m=0,1,2,3,\quad |x|>r_0. $$
Eventually we let $h=\sum_{k=1}^{\infty}\psi_k$ and all the derivatives of $h$ are small and decay at infinity, which means $u=U+h$ is convex.

The following lemma in \cite{bao-li-z1} proves that $c$ is uniquely determined by other parameters.

\begin{lem}\label{uniquenesslem}
Let $u_1$, $u_2$ be two locally convex smooth functions on $\mathbb R^2\setminus \bar D$ where $D$ satisfies the same assumption as in Theorem \ref{thm2}. Suppose
$u_1$ and $u_2$ both satisfy
$$\left\{\begin{array}{ll}\det(D^2u)=f\mbox{ in }~~ \mathbb R^2\setminus \bar D,\\
u=\varphi,\quad \mbox{ on }~~ \partial D
\end{array}
\right.
$$ with $f$ satisfying \eqref{af} and for the same constant $d$
\begin{equation}\label{12feb8e1}
u_i(x)-\frac 12|x|^2-d\log |x|=O(1), \quad x\in \mathbb R^2\setminus \bar D,\quad i=1,2.
\end{equation}
Then $u_1\equiv u_2$.
\end{lem}

Since Lemma \ref{uniquenesslem} uniquely determines the constant in the expansion, Theorem \ref{thm2} is established. $\Box$

\end{document}